\documentclass[12pt]{article}
\usepackage{amsmath,amssymb,amsfonts}  
\usepackage{amsthm}
\usepackage{graphicx,psfrag,epsf}
\usepackage{enumerate}
\usepackage{natbib}
\usepackage{url} 
\usepackage{mathptmx}      
\usepackage{microtype}
\usepackage{xspace}
\newcommand{\blind}{0}

\usepackage{graphicx}
\usepackage{mathptmx}      
\usepackage{tikz}
\usetikzlibrary{shapes}
\usepackage{amsmath,amssymb,amsfonts}         
\usepackage{url}

\usepackage{xspace}
\usepackage{microtype}

\usepackage{color}
\usepackage{dsfont}
\usepackage{bm}
\usepackage{algorithm}
\usepackage{algpseudocode}
\usepackage{mathrsfs}
\DeclareMathAlphabet{\mathantt}{OT1}{antt}{m}{n}

\graphicspath{{./pictures/},{./}}

\newcommand*{\LargerCdot}{{\raisebox{-0.5ex}{\scalebox{1.8}{$\cdot$}}}}

\newcommand*{\dist}{\ensuremath{\mathantt{K}}\xspace} 
\newcommand*{\rand}[1]{\ensuremath{\mathbf{#1}}\xspace} 
 
\newcommand*{\dens}{\ensuremath{\mathantt{p}}\xspace} 
\newcommand*{\densk}{\raisebox{0mm}{\ensuremath{\mathantt{k}}}\xspace} 
\newcommand*{\moveprob}{\ensuremath{\mathantt{q}}\xspace} 
 
\newcommand*{\acc}{\ensuremath{\mathantt{a}}\xspace} 
\newcommand*{\bcc}{\ensuremath{\mathantt{b}}\xspace} 
\newcommand*{\ratio}{\ensuremath{\mathantt{r}}\xspace} 
\newcommand*{\prob}[1]{\ensuremath{\mathbb{P}\left(#1\right)}\xspace} 
\newcommand*{\expec}[1]{\ensuremath{\mathbb{E}\left\{#1\right\}}\xspace} 
\newcommand{\ind}[1]{\raisebox{-0.3mm}{\scalebox{1.2}{\ensuremath{\mathds{1}}}}\ensuremath{ \left\{ #1\right\}}}

\newcommand*{\bmid}{\ensuremath{\ \Big\vert\ }} 
\newcommand*{\bbmid}{\ensuremath{\ \Bigg\vert\ }} 
 
\newcommand*{\rel}{\ensuremath{\mathfrak{R}}}

\newcommand*{\dimens}[1]{\ensuremath{\text{dim}\left(#1\right)}\xspace}

\newcommand*{\cp}{changepoint\xspace}
\newcommand*{\cps}{changepoints\xspace}
\newcommand*{\states}{\ensuremath{\mathcal{S}}\xspace}
\newcommand*{\ustates}{\ensuremath{\mathcal{U}}\xspace}
\newcommand*{\moves}{\ensuremath{\mathcal{M}}\xspace}
\newcommand*{\spaces}{\ensuremath{\mathcal{D}}\xspace}
\newcommand*{\mcal}[1]{\ensuremath{\mathcal{#1}}\xspace}

\newcommand*{\bo}[1]{\ensuremath{\mathbf{#1}}\xspace}

\newcommand*{\semic}{\ensuremath{\mathbin{\boldsymbol;}}\xspace} 
\newcommand*{\ellr}{{\ensuremath{r_\ell}}\xspace} 
\newcommand*{\kr}{{\ensuremath{r_k}}\xspace} 
 
\newcommand*{\todo}[1]{}

\newcommand*{\funcg}{\ensuremath{\mathantt{g}}\xspace}

\makeatletter
\newcommand*{\defeq}{\mathrel{\rlap{\raisebox{0.3ex}{$\m@th\cdot$}}\raisebox{-0.3ex}{$\m@th\cdot$}}=}
\makeatother
{\bf}{\rm}
\newtheorem{algo}{Algorithm}{\bf}{\rm}
\newtheorem{myremark}{Remark}{\bf}{\rm}
\newtheorem{definition}{Definition}{\bf}{\rm}
{\bf}{\rm}
\newtheorem{lemma}{Lemma}{\bf}{\rm}
{\bf}{\rm}
{\bf}{\rm}
\newtheorem{corollary}{Corollary}{\bf}{\rm}
{}{}
\begin{document}

\def\spacingset#1{\renewcommand{\baselinestretch}%
	{#1}\small\normalsize} \spacingset{1}


\if0\blind
{  
	\title{\bf A note on the Metropolis-Hastings acceptance probabilities for mixture spaces}
	\author{Tobias Siems\thanks{The author gratefully acknowledges support from the Landesgraduiertenf\"orderung, Greifswald.}
		, Lisa Koeppel \\ \smallskip
		Department of Mathematics and Computer Science\\
		University of Greifswald\\\\
	\underline{\textbf{This document is in draft stage and hasn't been submitted to a journal yet!}}
}
	\maketitle
} \fi

\if1\blind
{
	\bigskip
	\bigskip
	\bigskip 
	\begin{center}
		{\LARGE\bf Title}
	\end{center}
	\medskip
} \fi
\bigskip
\begin{abstract}
This work is driven by the ubiquitous dissent over the abilities and contributions of the Metropolis-Hastings and reversible jump algorithm within the context of trans dimensional sampling.
We demystify this topic by taking a deeper look into the implementation of Metropolis-Hastings acceptance probabilities with regard to general mixture spaces.
Whilst unspectacular from a theoretical point of view, mixture spaces gave rise to challenging demands concerning their effective exploration.
An often applied but not extensively studied tool for transitioning between distinct spaces are so-called translation functions.
We give an enlightening treatment of this topic that yields a generalization of the reversible jump algorithm and unveils another promising translation technique.
Furthermore, by reconsidering the well-known Metropolis within Gibbs paradigm, we come across a dual strategy to develop Metropolis-Hastings samplers.
We underpin our findings and compare the performance of our approaches by means of a change point example.
Thereafter, in a more theoretical context, we revitalize the somewhat forgotten concept of maximal acceptance probabilities.
This allows for an interesting classification of Metropolis-Hastings algorithms and gives further advice on their usage.
A review of some errors in reasoning that have led to the aforementioned dissent concludes this paper.
\end{abstract}

\noindent%
\vfill

\newpage
\spacingset{1.45} 
\sloppy
\section{Introduction}

The foundation of Markov chain Monte Carlo (MCMC) sampling is that under some circumstances Markov chains converge towards their invariant distribution regardless of the initial state \citep{tierney1994}. 
Therefore, MCMC methods provide schemes to build Markov kernels with a desired invariant distribution which is typically an intractable marginal or conditional, or highly multivariate distribution.

The Metropolis-Hastings algorithm is one of the most well-known MCMC methods.
It traverses through the state space by means of a user defined proposal distribution.
Each proposed state undergoes an accept-reject step which decides whether the proposed state or the previous link in the chain is chosen to be the next link.
This step alone secures the invariance of the Metropolis-Hastings Markov kernel towards the target distribution.



Originally, Metropolis-Hastings proposals were designed conveniently through kernel densities w.r.t. the same measure that runs the density of the target distribution \citep{metropolis1953, hastings1970}.
In this case, the acceptance probability used in the accept-reject step is determined by the likelihood ratio w.r.t. the transition in backward and forward direction.
However, upcoming applications like variable selection \citep{spike_slab_mitchel}, point processes \citep{geyer1994} and change point analysis \citep{exact_fearnhead, siems18} raised higher demands.

Common for these applications is that the elements of the state spaces are inhomogeneous in their dimension, i.e. the spaces are trans dimensional.
Obviously, such mixtures of different spaces are not straightforwardly accessible through standard techniques like random walk proposals.
The first methods to conduct a change in dimension were plain births and deaths \citep{geyer1994}.
However, this is prone to ignore the relations shared among several coordinates and therewith promotes poor acceptance rates.
Consequently, the exploration of the entire state space performs differently within the same and across the dimensions, which impairs the overall mixing time.

Subsequently, it became utterly popular to utilize functions that convey between points of different dimensions.
\cite{green1995} pioneered this approach through the reversible jump algorithm.
It was developed for purely continuous spaces and describes a particular class of proposals that act detached from the target space.
As a result, the density of the target distribution and the kernel density of the proposal do not necessarily share the same underlying measure anymore.
Nevertheless, \cite{green1995} was able to maintain the invariance of the target distribution through a nifty application of the change of variable theorem.

According to Google scholar, \cite{green1995} has been cited over 5000 times until now.
Unfortunately, \cite{green1995} and others caused a misperception about the abilities of the Metropolis-Hastings algorithm that is nowadays ubiquitous.
It is agreed among a significant number of scientific writings that the reversible jump algorithm has somehow made trans dimensional sampling possible and that other MCMC algorithms like Metropolis-Hastings are not applicable in these scenarios.
\cite{green1995}, \cite{carlin1995} and \cite{chen2012} even substantiate this wrong conclusion, which was acknowledged by \cite{besag2001markov, waagepetersen2001tutorial, green2003, green2009, sisson2005, sambridge2006} and many others.
A thorough search for papers which oppose this stance indirectly, in one way or another, revealed only \cite{geyer1994, tierney1998, andrieu2001model, godsill2001, jannink2004} and \cite{roodaki2011}.

This poses a significant division within the MCMC community that requires a thorough treatment.
Furthermore, the outshining popularity of the reversible jump algorithm makes this inspiring statistical area seem monotonous and boring.
Therefore, this paper takes up on the task of exploring general mixture spaces by means of the Metropolis-Hastings paradigm, whereby the very focus lies on the computation of acceptance probabilities.
We will see that the foundation of this computation in its most theoretical to most practical form is the compatibility of the involved densities towards a common underlying measure.

For example, the generalized Metropolis-Hastings algorithm from \cite{tierney1998}, that is applicable to virtually any combination of target distribution and proposal, derives the required densities through the Radon-Nikodym theorem \citep{nikodym1930generalisation}.
Unfortunately, the nonconstructive nature of this theorem only gives an approach of little practical value.

We will see that, subject to the aforementioned compatibility, the design of proposals comprises a trade-off between mutability and feasibility.
Thus, each of the following methods exhibits its very own conditions and possibilities that need to be considered carefully in order to build the right Metropolis-Hastings sampler.

As has already been done before, we examine so-called translation functions that convey between pairs of spaces.
However, we consider two different ways to apply them: before and after proposing new states and refer to these concepts as ad-hoc and post-hoc translations, respectively.

In accordance with the reversible jump algorithm, the main feature of post-hoc translations is the detachment of the proposal from the space the target distribution acts on.
However, this entails an integral transformation which imposes strict conditions and requirements on the translation functions.
Consequently, a sophisticated design in practical and mathematical terms is essential here.

In return, post-hoc translations support parsimony regarding the number of random variables that need to be generated.
This enables very tightly adapted and even deterministic transitions.
Furthermore, they grant access to certain proposals that define intractable compound distributions like convolutions, marginal distributions or factor distributions.

In contrast, ad-hoc translations are much less restricted and support arbitrary translation functions.
We may for example exploit the likelihood of the target distribution to improve the acceptance rates of difficult transitions.
However, they require that the proposal kernel densities comply with the underlying measure space.
Thus, the freedom in choosing translation functions comes at the prize of a smaller adaptability in terms of the generation of random states.

Whilst the name ad-hoc results from the ability to use simple and purposive translation functions, the term post-hoc refers to the moment when the translation function is applied.


%
%

The well-known Metropolis within Gibbs approach confines the set of possible transitions through conditioning 
and is therewith able to steer the exploration of the state space in specific ways.
Consequently, it employs conditional densities defined on suitable measure spaces.
It turns out that proposals set up directly on these measure spaces yield straightforwardly to compute acceptance probabilities.

This shifts the challenges of creating a Metropolis-Hastings sampler to deriving conditional densities of the target distribution.
We consider this in some way novel and innovative perspective to be dual to the traditional one that puts the definition of proposals first.

Metropolis within Gibbs is applicable in the usual way for coordinate spaces and therewith capable of relaxing the terms of ad-hoc translations.
However, it also captures far more difficult cases like so-called semi deterministic translations.

Our methods are scrutinized by means of a \cp example.
For this sake, we compare the acceptance rates of several birth and death like proposals.
It turns out that sophisticated ad-hoc proposals (combined with Metropolis within Gibbs) or post-hoc proposals are vital here.
While the post-hoc translations are derived from delicate properties of the target distribution, the ad-hoc translations just trace for high likelihood values.
Despite their differences, both methods achieve comparable acceptance rates.

The notion of maximal acceptance probabilities as introduced by \cite{peskun1973optimum} and \cite{tierney1998} doesn't seem to have gained a lot of attention so far.
The reason for this might be that there is usually only one choice for the acceptance probability in place.
We revitalize this dusted property and show that post-hoc translations do not necessarily yield maximal acceptance probabilities.
Therewith, they may sometimes differ from their maximal counterpart.

Transitions between different pairs of spaces are usually performed by separate proposals which are combined to a mixture proposal.
This allows for a clear modular design of proposals and is thus pursued primarily in this paper.

As stated in \cite{tierney1998} already, acceptance probabilities can either be computed based on unique pairs of the individual proposals or from the mixture proposal as a whole.
Albeit unintentionally, \cite{roodaki2011} elaborates conditions upon when both approaches yield maximality.
We will investigate their main result and put them into our context.

This paper is structured as follows.
At first, we consider the Metropolis-Hastings approach in applied terms in Section \ref{chap:mh}.
This involves an introduction of the important detailed balance condition together with the primal Metropolis-Hastings algorithm.
Furthermore, in Section \ref{chap:mixspaces} we talk about mixture spaces, mixture proposals and translations.
Section \ref{chap:metgibs} is concerned with the Metropolis within Gibbs approach.
Subsequently, we elaborate a \cp example in Section \ref{chap:cpexample}.
The general consideration of Metropolis-Hastings and maximal acceptance probabilities is pursued in Section \ref{chap:genmh}.
Section \ref{chap:mixtures} further transfers these observations to mixture proposals.
We conclude with a discussion and a brief literature review in Section \ref{chap:discussion}.

\section{The Metropolis-Hastings algorithm in applied terms}
\label{chap:mh}

In this section, we deal with practical implementations of the Metropolis-Hastings algorithm.
To this end, we develop ways to design proposals based on kernel densities w.r.t. sigma finite measure spaces.
These comprise so-called ad-hoc and post-hoc translations and a new perspective for the Metropolis within Gibbs framework.

The Metropolis-Hastings algorithm is usually introduced on the basis of densities w.r.t. a common measure.
For this sake, we are given a sigma finite measure space $(\states,\mcal{A},\lambda)$ and a probability measure $\pi$ with a density $\dens$ such that $\pi(ds)=\dens(s)\lambda(ds)$.
The user provides a Markov kernel $\dist$ from $(\states,\mcal{A})$ to $(\states,\mcal{A})$
in form of a kernel density $\densk(s', s)$ such that $\dist(s', ds)=\densk(s',s)\lambda(ds)$.
This Markov kernel is referred to as the \emph{proposal} and $\densk$ as the \emph{proposal kernel density}.

\begin{algo}[Metropolis-Hastings]
	\label{alg:mh}
	(I) Choose an initial state $s_0\in\states$ (II) In step $k=1,2,....$, given the previous state $s_{k-1}=s'$, propose a new state $s$ according to $\dist(s', \LargerCdot)$ and set $s_k=s$ with probability
	\begin{align}
	\label{eq:acc}
	\acc_{s' s}=\min\Bigg\{1, \frac{\dens(s)\densk(s,s')}{\dens(s')\densk(s',s)}\Bigg\}
	\end{align}
	otherwise $s_k=s'$. We set $\acc_{s', s}$ to zero wherever it is not defined.
\end{algo}
We refer to (\ref{eq:acc}) as the \emph{acceptance probability} and to the second argument within the braces as the \emph{acceptance ratio}.

Let further
\begin{align}
\label{eq:mhtranskern}
\mu(s',\LargerCdot)=\int \acc_{s' s}\delta_{s}(\LargerCdot)+\left(1-\acc_{s's}\right)\delta_{s'}(\LargerCdot)\dist(s', ds)
\end{align}
$\mu$ represents the Markov kernel whose application constitutes step (II) of Algorithm \ref{alg:mh}.	
Consequently, executing Algorithm \ref{alg:mh} corresponds to simulating successive links of a Markov chain with initial state $s_0$ and Markov kernel $\mu$.

The crucial point here is that the distributions of these links converge to $\pi$, subject to some conditions \citep{tierney1998}.
Thus, after simulating the chain for a sufficiently long time, the last link can be regarded as an approximate sample of $\pi$.

For us, the most important of these conditions is that $\pi$ has to be an invariant distribution of $\mu$.
Thus, in the following, we will elaborate different acceptance probabilities and Metropolis-Hastings Markov kernels, whereby our goal will always be to show the invariance of the target distribution towards the Metropolis-Hastings Markov kernel.
These proves are facilitated by means of the following condition.
\begin{definition}
	The Markov kernel $\mu$ preserves the \emph{detailed balance} condition w.r.t. $\pi$ if
	\begin{align*}
	\pi\otimes \mu(A\times B)=\pi\otimes \mu(B\times A)
	\end{align*}
	for all $A,B\in\mcal{A}$. 
\end{definition}

If $\mu$ preserves the detailed balance condition w.r.t. $\pi$, $\pi$ is an invariant distribution of $\mu$ since $\pi\otimes \mu(\Omega\times \LargerCdot)=\pi\otimes \mu(\LargerCdot\times \Omega)=\pi$.
The opposite implication does not hold in general.

Markov chains build from Markov kernels that preserve the detailed balance w.r.t. another distribution are called \emph{reversible}.
This is because, they exhibit same probabilities in forward and backward direction once a link in the chain follows the law of this distribution.
Consequently, MCMC methods that preserve the detailed balance condition are also called reversible.

\begin{lemma}[\cite{hastings1970}]
	\label{lem:mh}
	The Markov kernel $\mu$ of Equation (\ref{eq:mhtranskern}) with the acceptance probability shown in (\ref{eq:acc})  preserves the detailed balance w.r.t. $\pi$.
\end{lemma}
\begin{proof}
	There is nothing to prove for $A,B\in\mcal{A}$ with $A=B$. 
	Since $A\times B$ is the disjoint union of $(A\cap B)\times (A\cap B)$, $(A\backslash B)\times (B\backslash A)$, $(A\backslash B)\times (A\cap B)$ and $(A\cap B)\times (B\backslash A)$ we can w.l.o.g. assume that $A\cap B=\emptyset$.
	We get
	\begin{align*}
	&\pi\otimes \mu(A\times B)=\int_A \mu(s',B)\pi(ds')=\int_A \int_B \acc_{s' s}\densk(s',s)\dens(s')\lambda(ds)\lambda(ds')\\
	&\overset{*}{=}\int_B \int_A  \acc_{s s'}\densk(s,s')\dens(s)\lambda(ds')\lambda(ds)=\int_B \mu(s,A)\pi(ds)=\pi\otimes \mu(B\times A)
	\end{align*}
	whereby we have used that $\acc_{ s's}\densk(s',s)\dens(s')=\acc_{s s'}\densk(s,s')\dens(s)$ and Fubini's theorem in *.
\end{proof}

By looking at $\acc_{s' s}$, we see that normalization constants w.r.t. $\dens$ cancel out.
Thus, the algorithm may deal with conditional versions of $\dens$ by deploying $\dens$ unchanged.
This greatly facilitates data processing by means of conditioning.

\subsection{Mixture state spaces and translations}
\label{chap:mixspaces}
Now, we elaborate sampling across alternating measurable spaces with the help of functions that translate between these spaces.
We consider two naturally arising approaches, the so-called ad-hoc and post-hoc translations.
Both exhibit their very own challenges and advantages in terms of feasibility and adaptability.

Let $\spaces$ be a finite or countable set and let $(\states_i,\mcal{A}_i, \lambda_i)_{i\in\spaces}$ be a family of disjoint sigma finite measure spaces.
The \emph{mixture measure space} $(\states,\mcal{A}, \lambda)$ of $(\states_i,\mcal{A}_i, \lambda_i)_{i\in\spaces}$ is defined through 
\begin{align*}
&\states=\bigcup_{i\in\spaces}\states_i,\quad \mcal{A}=\sigma\Bigg(\bigcup_{i\in\spaces}\mcal{A}_i\Bigg),\quad \lambda=\sum_{i\in\spaces}\lambda_i\circ \text{id}_i^{-1}
\end{align*}
whereby $\mcal{A}$ is the smallest sigma algebra that comprises all $\mcal{A}_i$'s.
The sole purpose of the identities $\text{id}_i:\states_i\rightarrow \states$ with $\text{id}_i(s)=s$ is to lift states from the component spaces into the mixture space.

As before, we assume that $\pi$ exhibits a density $\dens$ w.r.t. $\lambda$.
A Metropolis-Hastings algorithm for $\pi$ can now be build as usual.
A natural way to set up proposals on the mixture space is by pursuing a modular design through a mixture of Markov kernels, whereby each component solely transitions within a pair of the spaces.
We  want to use the term \emph{move} to feature the available transitions.

Let $\moves$ be a finite or countable set, i.e. the set of \emph{moves}.
In order to build a mixture proposal, we define measurable functions $\moveprob_\ell:\states\rightarrow[0,1]$ with $\sum_{\ell\in\moves}\moveprob_{\ell}(s')=1$ for all $s'\in\states$.
$\moveprob_{\ell}(s')$ is the probability of choosing move $\ell$ while being in $s'$.
Furthermore, each move $\ell$ exhibits a kernel density $\densk_\ell$ that transitions between a pair of the spaces and complies with the measure of the target space.

For the sake of technical correctness, we need to lift each $\densk_\ell$ into the mixture space by setting it to 0 for unsupported arguments.
Therewith, the mixture proposal $\dist$ that operates on the mixture space $(\states,\mcal{A})$ reads
\begin{align*}
\dist(s',ds)=\sum_{\ell\in\moves}\moveprob_\ell(s')\dist_\ell(s',ds)
\end{align*}
with $\dist_\ell(s',ds)=\densk_\ell(s',s)\lambda(ds)$.

This time, the transition kernel for the Metropolis-Hastings algorithm for mixtures reads
\begin{align}
\label{eq:mhtranskernmix}
\mu(s',\LargerCdot)=\sum_{\ell\in\moves}\moveprob_\ell(s')\int \acc_{s' s}^\ell\delta_{s}(\LargerCdot)+\left(1-\acc_{s's}^\ell\right)\delta_{s'}(\LargerCdot)\dist_\ell(s', ds)
\end{align}
Given the previous link $s'$, this kernel involves choosing a move $\ell$ according to $\left(\moveprob_k(s')\right)_{k\in\moves}$ at first, then proposing an $s$ according to $\dist_\ell(s',\LargerCdot)$ and finally accepting or rejecting $s$ through $\acc_{s's}^\ell$.

This is in accordance with Algorithm \ref{alg:mh}, however, with the difference that we do not consider the mixture proposal as a whole.
Instead, we incorporate the moves explicitly into the acceptance probability (see also Section \ref{chap:mixtures}).

\begin{lemma}[Siems]
	\label{lem:mix}
	Let $r_\ell:\moves\rightarrow\moves$ be a bijection.
	The transition kernel $\mu$ of (\ref{eq:mhtranskernmix}) with acceptance probability 
	\begin{align*}
	\acc_{s' s}^\ell=
	\min\Bigg\{1,\frac{\dens(s)\densk_{\ellr}(s,s')}{\dens(s')\densk_{\ell}(s',s)}\cdot\frac{\moveprob_{\ellr}(s)}{\moveprob_{\ell}(s')}\Bigg\}
	\end{align*}
	preserves the detailed balance w.r.t. $\pi$.
\end{lemma}
\begin{proof}
	For $A\in\mcal{A}_j$, $B\in\mcal{A}_i$, we get
	\begin{align*}
	&\sum_{\ell\in\moves}\int_A\int_B\acc_{s' s}^\ell\dens(s')\moveprob_{\ell}(s')\densk_{\ell}(s',s)\lambda_i(ds)\lambda_j(ds')\\
	&=\sum_{\ell\in\moves}\int_B\int_A\acc_{s s'}^\ellr\dens(s)\moveprob_{\ellr}(s)\densk_{\ellr}(s,s')\lambda_j(ds')\lambda_i(ds)\\
	&\overset{*}{=}\sum_{\ell\in\moves}\int_B\int_A\acc_{s s'}^\ell\dens(s)\moveprob_{\ell}(s)\densk_{\ell}(s,s')\lambda_j(ds')\lambda_i(ds)
	\end{align*}
	whereby in * we have used the uniqueness of $r_\ell$. 
\end{proof}

$\ellr$ stands for the unique backward or reverse move of $\ell$.
Consequently, if move $\ell$ transitions from $\states_j$ to $\states_i$, move $\ellr$ should transition in the opposite direction from $\states_i$ to $\states_j$.	
Please note that the uniqueness of $\ellr$ doesn't imply that the transitions performed by move $\ell$ can only be reversed by move $\ellr$.
In fact, the pairs of spaces are not supposed to be unique to the moves and the possible transitions determined by related move pairs may overlap arbitrarily.

Concerning the design of the moves, it is often not clear how to transition away from a point in one space to a point in another space.
Especially random walk proposals pose a problem if there is no suitable measure of distance between the spaces available.
This gives rise to two distinct approaches that employ functions as a device for translation.

\begin{figure}[ht]
	\center
	\begin{tikzpicture}[scale=0.80]
	\tikzstyle{every node}=[draw,shape=ellipse, line width=1.1pt];
	\node (node1) at (150:3) {$\states_j$};
	\node (node2) at ( 30:3) {$\states_i$};
	\node (node3) at ( 270:0)[rectangle] {$\densk_{\ell}$};
	\node (node4) at ( 270:1.3)[] {$\states_i$};
	\node (node5) at ( 270:-3)[rectangle] {$\bo s_{\ell}$};
	\node[font=\large, draw=none,fill=none] at (270:-4) {(a)};

	\path[->] (node1.north) edge [out=50, in=180, line width=1.1pt] (node5.west);
	\path[->] (node5.east) edge [out=0, in=120, line width=1.1pt] (node2.north);
	\path[->] (node2.south) edge [out=250, in=0, line width=1.1pt] (node3.east);
	\path[->] (node3) edge [, line width=1.1pt] (node4.north);
	\end{tikzpicture}
	$\quad\quad\quad\quad$
	\begin{tikzpicture}[scale=0.80]
	\tikzstyle{every node}=[draw,shape=ellipse, line width=1.1pt];
	\node (node1) at (150:3) {$\states_j$};
	\node (node2) at ( 30:3) {$\ustates_{\ell}$};
	\node (node3) at ( 270:0)[rectangle] {$\bo s_{\ell}$};
	\node (node4) at ( 270:1.3) {$\states_i$};
	\node (node5) at ( 270:-3)[rectangle] {$\densk_{\ell}$};
	\node[font=\large, draw=none,fill=none] at (270:-4) {(b)};

	\path[->] (node1) edge [out=-20, in=120, line width=1.1pt] (node3);
	\path[->] (node1.north) edge [out=50, in=180, line width=1.1pt] (node5.west);
	\path[->] (node5.east) edge [out=0, in=120, line width=1.1pt] (node2.north);
	\path[->] (node2) edge [out=200, in=60, line width=1.1pt] (node3);
	\path[->] (node3) edge [, line width=1.1pt] (node4.north);
	\end{tikzpicture}
	\caption{Transition graphs for a move $\ell\in\moves$ that transitions from $\states_j$ to $\states_i$ in the fashion of (a) an ad-hoc translation and (b) a post-hoc translation}
	\label{fig:mhrjcomp}
\end{figure}
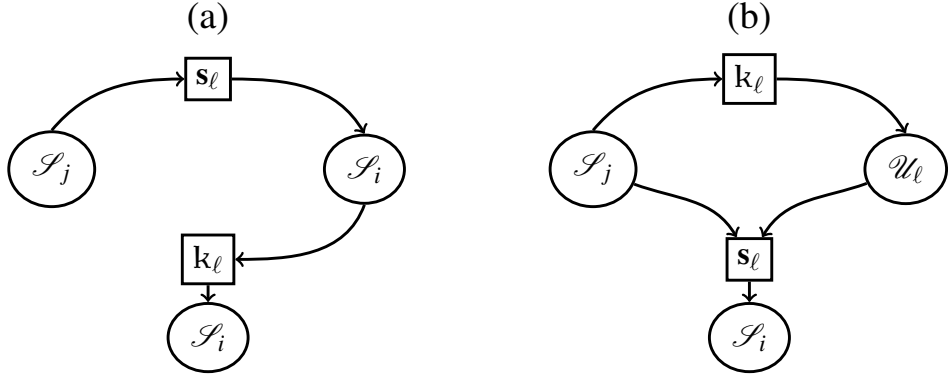

In order to transition from $\states_j$ to $\states_i$ in move $\ell$, our first approach employs a measurable function $\bo s_{\ell}:\states_j\rightarrow\states_i$ that translates between the spaces and a kernel density $\densk_\ell$ from $(\states_i,\mcal{A}_i)$ to $(\states_i,\mcal{A}_i, \lambda_i)$.
Functions like $\bo s_{\ell}$ are referred to as \emph{translation functions}.
Given the previous link in the chain $s'\in\states_j$, new states are proposed by virtue of $\dist_\ell(s',\LargerCdot)=\int_\LargerCdot\densk_{\ell}(\bo s_{\ell}(s'), s)\lambda_i(ds)$.
We call this procedure \emph{ad-hoc translation}.
Figure \ref{fig:mhrjcomp}(a) illustrates it graphically.

\begin{corollary}[Siems]
\label{cor:ad-hoc}
The transition kernel $\mu$ of (\ref{eq:mhtranskernmix}) with acceptance probability 
\begin{align*}
\acc_{s' s}^\ell=
\min\Bigg\{1,\frac{\dens(s)\densk_{\ellr}(\bo s_{\ellr}(s),s')}{\dens(s')\densk_{\ell}(\bo s_{\ell}(s'),s)}\cdot\frac{\moveprob_{\ellr}(s)}{\moveprob_{\ell}(s')}\Bigg\}
\end{align*}
preserves the detailed balance w.r.t. $\pi$ in ad-hoc translations.
\end{corollary}

Ad-hoc translations pose a flexible framework that is directly accessible through non-mathematicians with a basic understanding of the Metropolis-Hastings algorithm.
In practice, however, another technically more demanding Metropolis-Hastings method has become the quasi gold-standard for sampling within mixed spaces.
It is commonly referred to as the reversible jump algorithm \citep{green1995}.
We will allocate this approach to another type of translation and give a novel result that slightly generalizes this technique.

In contrast to ad-hoc translations, now, we propose first and then apply a translation to the proposed state.
This allows us to detach the target space from the space where the proposal kernel density acts on.
Thus, in move $\ell\in\moves$ we may transition between a pair of the spaces, say from $\states_j$ to $\states_i$, by detouring over another user-defined sigma finite measure space $(\ustates_{\ell},\mcal{B}_{\ell},\nu_{\ell})$.
To this end, we design a kernel density $\densk_{\ell}$ from $(\states_j,\mcal{A}_j)$ to $(\ustates_{\ell},\mcal{B}_{\ell},\nu_{\ell})$ and a measurable function $\bo s_{\ell}:\states_j\times \ustates_{\ell}\rightarrow \states_i$.
Given the previous link $s'\in\states_j$, a random state $u'\in \ustates_{\ell}$ is drawn by virtue of $\int_\LargerCdot\densk_{\ell}(s', u')\nu_{\ell}(du')$ and the newly proposed state $s\in\states_i$ is obtained as $s=\bo s_{\ell}(s',u')$.
We call this procedure \emph{post-hoc translation}.
Figure \ref{fig:mhrjcomp}(b) illustrates it graphically.

The case where $(\states_j,\mcal{A}_j)$ and $(\states_i,\mcal{A}_i)$ are discrete can be treated straightforwardly since the acceptance probability of move $\ell$ then reads
\begin{align}
\label{eq:accdiscpost}
\acc_{s' s}^\ell=
\min\Bigg\{1,\frac{\dens(s)\int\delta_{\bo s_\ellr(s,u)}\left(\{s'\}\right)\densk_{\ellr}(s,u)\nu_{\ellr}(du)}
{\dens(s')\int\delta_{\bo s_\ell(s',u')}\left(\{s\}\right)\densk_{\ell}(s',u')\nu_{\ell}(du')}\cdot\frac{\moveprob_{\ellr}(s)}{\moveprob_{\ell}(s')}\Bigg\}
\end{align}
This applies, because the proposals in forward and backward direction are summarized as discrete and are thus compatible to $\dens$.
Though, the required feasibility of the integrals in (\ref{eq:accdiscpost}) slightly restricts the set of possible post-hoc translations here.

However, most interesting cases arise when the proposal $\int \delta_{\bo s_\ell(s',u')}(\LargerCdot)\densk_\ell(s',u')\nu_{\ell}(du')$ either doesn't exhibit a kernel density w.r.t. $\lambda_i$ or when it is computationally infeasible.
Unfortunately, our methods so far do rely on such densities which calls for a different approach here.

In practice, it would be desirable to derive the acceptance probabilities directly from the given proposal kernel densities.
Thus, the trick is to accept and reject $s=\bo s_{\ell}(s',u')$ within $\states_j\times\ustates_\ell$, which yields the actual post-hoc paradigm.
In order for this procedure to be well-defined, any transition $(s',u')\in \states_j\times\ustates_\ell$ must exhibit a unique backward transition $(s,u)\in \states_i\times\ustates_\ellr$.
Thus, we need to define additional measurable functions $\bo u_\ell:\states_j\times\ustates_\ell\rightarrow \ustates_\ellr$ with
\begin{align*}
(\bo s, \bo u)_\ell(s',u')=(s,u)\quad \Leftrightarrow\quad (\bo s, \bo u)_\ellr(s,u)=(s',u')
\end{align*}
for all $(s',u')\in\states_j\times\ustates_\ell$ and $(s,u)\in\states_i\times\ustates_{r_\ell}$.


This reasoning shows that the possible options for choosing post-hoc translations this way are somewhat limited.
Therewith, the seemingly restrictive conditions of the next lemma can be considered as fairly universal.

\begin{lemma}[Siems]
	\label{lem:post-hoc}
	Assume that for move $\ell\in\moves$ with transitions from $\states_j$ to $\states_i$ and corresponding unique backward move $r_\ell$, there exist measurable functions $\bo u_{\ell}:\states_j\times \ustates_{\ell}\rightarrow \ustates_{\ellr}$ and $f_{\ell}:\states_j\times \ustates_{\ell}\rightarrow\mathbb{R}_{\geq 0}$ with
	\begin{align*}
	&\textbf{\emph{(I)}}\ \ (\bo s, \bo u)_\ell=(\bo s, \bo u)_\ellr^{-1}\\
	&\textbf{\emph{(II)}}\ \ \lambda_i\otimes\nu_{\ellr}\left((\bo s, \bo u)_{\ell}(ds',du')\right)=f_{\ell}(s',u')\lambda_j\otimes\nu_{\ell}(ds', du')
	\end{align*}
	The transition kernel
	\begin{align*}
	\mu(s',\LargerCdot)=\sum_{\ell\in\moves}\moveprob_\ell(s')\int \acc_{s'u' }^\ell\delta_{\bo s_\ell}(\LargerCdot)+\left(1-\acc_{s'u'}^\ell\right)\delta_{s'}(\LargerCdot)\densk_\ell(s',u')\nu_{\ell}(du')
	\end{align*}
    with
	\begin{align*}
	\acc_{s'u'}^\ell=\min\left\{1,\ \frac{\dens(\bo s_{\ell})\densk_{\ellr}((\bo s, \bo u)_{\ell})}{\dens(s')\densk_{\ell}(s',u')}\cdot\frac{\moveprob_{\ellr}(\bo s_{\ell})}{\moveprob_{\ell}(s')}\cdot f_{\ell}(s',u')\right\}
	\end{align*}
	preserves the detailed balance w.r.t. $\pi$ in post-hoc translations.
\end{lemma}
\begin{proof}
	\textbf{(I)} implies that each $(\bo s, \bo u)_{\ell}:\states_j\times \ustates_{\ell}\rightarrow\states_i\times \ustates_{\ellr}$ is a bijection with inverse $(\bo s, \bo u)_{\ellr}$ and \textbf{(II)} guides through an integral transformation by means of the densities $f_{\ell}$.
	Since 
	\begin{align*}
	&\lambda_j\otimes\nu_{\ell}(ds', du')=\lambda_j\otimes\nu_{\ell}(\bo s_{\ellr}(\bo s_{\ell}(ds', du')))\\
	&=f_{\ell}(\bo s_{\ell}(ds', du'))\lambda_i\otimes\nu_{\ellr}(\bo s_{\ell}(ds', du'))=f_{\ell}(\bo s_{\ell}(s', u'))f_{\ellr}(s',u')\lambda_j\otimes\nu_{\ell}(ds', du')
	\end{align*}
	we see that $f_{\ell}(s',u')f_{\ellr}(s,u)=1$ if $\bo s_{\ell}(s', u')=(s,u)$ and vice versa.
	Furthermore, for $A\in\mcal{A}_j$ and $B\in\mcal{B}_{i}$, we get
	\begin{align*}
	\pi\otimes\dist_\ellr(B\times A)
	&=\int\delta_{s}(B)\delta_{\bo s_{\ellr}}(A) \dens(s)\densk_{\ellr}(s,u)\lambda_i\otimes\nu_{\ellr}(ds,du)\\
	&=\int\delta_{\bo s_{\ell}}(B)\delta_{s'}(A) \dens(\bo s_{\ell})\densk_{\ellr}((\bo s, \bo u)_{\ell})\lambda_i\otimes\nu_{\ellr}((\bo s, \bo u)_{\ell}(ds',du'))\\
	&=\int \delta_{s'}(A)\delta_{\bo s_{\ell}}(B)\dens(\bo s_{\ell})\densk_{\ellr}((\bo s, \bo u)_{\ell})f_{\ell}(s',u')\lambda_j\otimes \nu_\ell(ds', du')
	\end{align*}
	The rest of the proof now follows the same scheme as before.	
\end{proof}

In mathematically precise terms, it is incorrect to refer to $\densk_\ell$ as a proposal kernel density because it is generally not a kernel density of the proposal.
However, in order to maintain a consistent language, we still want to keep this naming.
Thus, the proof of Lemma \ref{lem:post-hoc} applies an integral transformation that secures the compatibility of the (joint) densities w.r.t. the transitions in forward and backward direction.

If all spaces are discrete coordinate spaces, we may choose arbitrary bijections $(\bo s, \bo u)_{\ell}$ that meet \textbf{(I)} and apply
\begin{align*}
\acc_{s'u'}^\ell=\min\left\{1,\ \frac{\dens(\bo s_{\ell})\densk_{\ellr}((\bo s, \bo u)_{\ell})}{\dens(s')\densk_{\ell}(s',u')}\cdot\frac{\moveprob_{\ellr}(\bo s_{\ell})}{\moveprob_{\ell}(s')}
\cdot \frac{\lambda_i\otimes\nu_{\ellr}((\bo s, \bo u)_{\ell}(ds', du'))}{\lambda_j\otimes\nu_{\ell}(ds',du')}\right\}
\end{align*}
\citep{fronk2004} deploy post-hoc translation in this way.

It is common to require that $\dimens{\states_j}+\dimens{\ustates_{\ell}}=\dimens{\states_i}+\dimens{\ustates_{\ellr}}$ if move $\ell$ transitions from $\states_j$ to $\states_i$.
\cite{green1995} pioneered the post-hoc translation approach and introduced this condition as \emph{dimension matching}.
His intention was to embed the spaces into a single superordinate one in order to avoid the inhomogeneity of the dimension.
Around that time the nowadays ubiquitous term \emph{trans dimensional sampling} was coined.

The \emph{reversible jump algorithm} of \cite{green1995} employs post-hoc translations over purely continuous coordinate spaces.
It requires the dimension matching condition and that each $(\bo s, \bo u)_{\ell}$ is a diffeomorphism with Jacobi determinant $\bo J_{\ell}$.
If \textbf{(I)} is met, the resulting acceptance probability reads
\begin{align*}
\acc_{s'u'}^\ell=\min\left\{1,\ \frac{\dens(\bo s_{\ell})\densk_{\ellr}((\bo s, \bo u)_{\ell})}{\dens(s')\densk_{\ell}(s',u')}\cdot\frac{\moveprob_{\ellr}(\bo s_{\ell})}{\moveprob_{\ell}(s')}
\cdot |\bo J_{\ell}|\right\}
\end{align*}
This follows from an integration by substitution for multiple variables since $ds'du'|\bo J_{\ell}|=dsdu$ if we substitute $(s,u)=(\bo s, \bo u)_{\ell}(s', u')$.

Post-hoc translations are highly advanced, require very careful implementations and as it stands do only support certain kinds of state spaces.
In return, they give the opportunity to generate random states within spaces of lower dimension than the target space.
Therewith, they pave the way for tightly adapted, parsimonious sampling schemes where even deterministic transitions are viable.

Moreover, post-hoc translations may act on a space of higher dimension than the target space, which grands access to intractable compound distributions.
To see this, consider the following minimal example to ordinarily transition from $\mathbb{R}$ to $\mathbb{R}$ by proposing values in $\mathbb{R}^2$ through the kernel density $\densk$.
By employing the diffeomorphism $(s', u_1, u_2)\mapsto \left(u_1+u_2, s'-u_2, u_2\right)$ we end up with equal forward and backward moves and
\begin{align*}
\acc_{s', u_1, u_2}=\min\left\{1,\ \frac{\densk(u_1+u_2, s'-u_2, u_2)\dens(u_1+u_2)}{\densk(s',u_1, u_2)\dens(s')}\right\}
\end{align*}

The corresponding proposal on the target space, i.e. $\dist(\LargerCdot)=\int \delta_{u_1+u_2}(\LargerCdot)\densk(s, u_1, u_2)du_1du_2$, is the convolution of $\densk(s, u_1, u_2)$ w.r.t $u_1$ and $u_2$.
If $\dist$ exhibits an intractable density on $\mathbb{R}$, the post-hoc paradigm proves functional here.
Nevertheless, there is a little price to pay when we use the post-hoc paradigm this way:
We lose the maximality of the acceptance probability.
This topic will be discussed in Section \ref{chap:genmh}.

\subsection{Metropolis within Gibbs}
\label{chap:metgibs}

Imagine a coordinate state space, whereby each step in the Metropolis-Hastings algorithm solely modifies one single coordinate.
This idea leads to the Metropolis within Gibbs framework.
We generalize this approach in order to allow the identification of new measure spaces to run proposals on.
Therewith, we create a new perspective that even provides access to an important class of post-hoc translations.

Given a countable set of moves $\moves$, consider a family of measurable functions $(\funcg_\ell)_{\ell\in \moves}$ from $(\states, \mcal{A})$ into a measurable space $(\mcal{G}_\ell,\mcal{B}_\ell)$ with $\{g\}\in\mcal{B}_\ell$ for all $g\in\mcal{G}_\ell$.
Each $\funcg_\ell$ determines a partition of the state space $\states$ into measurable sets that yield the same values under $\funcg_\ell$. 
Given the previous link $s'\in\states$, the idea is to choose an $\ell$ out of $\moves$ and to propose a new state exclusively within the ensuing set 
\begin{align*}
\rel_\ell(s')=\{s\in\states\mid \funcg_\ell(s)=\funcg_\ell(s')\}
\end{align*}

This entails the use of conditional versions of $\pi$.
Let $\rand X$ be a random variable with $\rand X\sim \pi$ and $\prob{\rand X\in ds\mid \funcg_\ell(\rand X)=g}=\dens_\ell(g, s)\lambda_\ell(ds)$ with a measure $\lambda_\ell$ over $\mcal{A}$ and a kernel density $\dens_\ell$ from $(\mcal{G}_\ell,\mcal{B}_\ell)$ to $(\states,\mcal{A}, \lambda_\ell)$.
For all $s'\in\states$, the restriction of $\lambda_\ell$ to $\rel_\ell(s')$ is supposed to be sigma finite.

On these grounds, we employ proposal kernel densities $\densk_\ell$ from $(\states,\mcal{A})$ to $(\states,\mcal{A}, \lambda_\ell)$ 
such that $\dist_\ell(s',ds)=\densk_\ell(s',s)\lambda_\ell(ds)$
and $\dist_\ell(s',\rel_\ell(s'))=1$ for all $\ell\in\moves$ and $s'\in \states$.
Furthermore, we define move probabilities $\moveprob_\ell:\states\rightarrow[0,1]$ as before.

\begin{lemma}[Siems]	
	\label{lem:metgibbs}
    The transition kernel $\mu$ of (\ref{eq:mhtranskernmix}) with acceptance probability 
	\begin{align}
	\label{eq:metgibbsacc}
	\acc_{s' s}^\ell=
	\min\left\{1,\frac{\dens_\ell(g, s)\densk_\ell(s, s')}{\dens_\ell(g, s')\densk_\ell(s',s)}\cdot\frac{\moveprob_\ell(s)}{\moveprob_\ell(s')}\right\}
	\end{align}
	preserves the detailed balance w.r.t. $\pi$.
\end{lemma}
\begin{proof}
	Consider the following equality
	\begin{align*}
	\pi(\LargerCdot)=\int\int_{\LargerCdot}\dens_\ell(g, s')\lambda_\ell(ds')\pi\circ\funcg_\ell^{-1}(dg)
	\end{align*}
	For $A,B\in\mcal{A}$ with $A\cap B=\emptyset$ we get
	\begin{align*}
	&\pi\otimes\mu(A\times B)
	=\int\sum\limits_{\ell\in\moves}\int_A\int_B\acc_{s's}^\ell\moveprob_\ell(s')\densk_\ell(s',s) \lambda_\ell(ds)\dens_\ell(g, s')\lambda_\ell(ds')\pi\circ\funcg_\ell^{-1}(dg)\\
	&=\int\sum\limits_{\ell\in\moves}\int_A\int_B\ind{\funcg_\ell(s')=\funcg_\ell(s)=g}\acc_{s's}^\ell\moveprob_\ell(s')\densk_\ell(s',s) \lambda_\ell(ds)\dens_\ell(g, s')\lambda_\ell(ds')\pi\circ\funcg_\ell^{-1}(dg)\\
	&\overset{*}{=}\int\sum\limits_{\ell\in\moves}\int_B\int_A\acc_{ss'}^\ell\moveprob_\ell(s)\densk_\ell(s,s')\lambda_\ell(ds') \dens_\ell(g, s)\lambda_\ell(ds)\pi\circ\funcg_\ell^{-1}(dg)
	=\pi\otimes\mu(B\times A)
	\end{align*}
	whereby we have used Fubini's theorem together with the partial sigma finiteness of $\lambda_\ell$ in *. 
\end{proof}

The major application scenario for Metropolis within Gibbs is to simplify sampling within coordinate spaces.
Let $\states=\states_1\times...\times\states_n$ be an arbitrary coordinate space together with a product sigma algebra $\mcal{A}$ and for $\ell=1,...,n$ define $\funcg_\ell(s)=(s_1,...,s_{\ell-1},s_{\ell+1},...,s_n)$.
Move $\ell$ describes transitions where all but the $\ell$-th coordinate remain fixed.

Let further $\pi(ds)=\dens(s)\nu(ds)$ with a product measure $\nu=\nu_1\otimes...\otimes\nu_n$ over $\mcal{A}$. We define
\begin{align*}
&\dens_\ell(\funcg_\ell(s'), s)=\frac{\dens(s)}{\int\dens(s)\nu_\ell(ds_\ell)}\ind{\funcg_\ell(s')=\funcg_\ell(s)}\quad\text{ and }\quad\lambda_\ell(ds)=\nu_\ell(ds_\ell)\prod_{j\not =\ell}\#(ds_j)
\end{align*} 
with the counting measure $\#$. 
$\int\dens(s)\nu_\ell(ds_\ell)$ is negligible here since it cancels out in (\ref{eq:metgibbsacc}).
$\dist_\ell$ is essentially build from a kernel density w.r.t. $\nu_\ell$.
If we chose plain probabilities $\moveprob_\ell$ for $\ell\in\moves$, we may write 
$\acc_{s' s}^\ell=\min\left\{1,\frac{\dens(s)\densk_\ell(s, s')}{\dens( s')\densk_\ell(s', s)}\right\}$
for transitions $(s',s)$ that differ in coordinate $\ell$ only.

The Metropolis within Gibbs approach exhibits a tremendous potential and even captures particular post-hoc translations.
To see this, consider a translation function $\bo s:\states'\rightarrow\states$ between two measure spaces $(\states',\mcal{A}',\lambda')$ and $(\states,\mcal{A},\lambda)$. 
The corresponding post-hoc translation from $\states'$ to $\states$ applies $\bo s$ in a deterministic manner.
However, if $\bo s$ is not injective, the backward move is usually not deterministic.
In the following, we refer to this sort of moves as semi deterministic translations (SDT's).
In practice, SDT's are the most commonly applied post-hoc translations.

Define
\begin{align*}
\funcg(s)=\begin{cases}
\bo s(s)& s\in\states'\\
s& s\in\states
\end{cases}
\end{align*}
This yields Metropolis within Gibbs moves that mimic post-hoc translations.
In order to transition from $\states$ to $\states'$ we rely on the concrete form of $\prob{\rand X\in \LargerCdot\mid \funcg(\rand X)=g}$.
Even though transitions within one space and the same are not per se excluded here, we deem them as unsupported in this particular move.

In contrast to Lemma \ref{lem:post-hoc}, $\bo s$ is not subject to restrictions.
Lemma \ref{lem:metgibbs} ensures that the acceptance probabilities are correct as long as we work within measure spaces that are in line with $\prob{\rand X\in \LargerCdot\mid \funcg(\rand X)=g}$.
Nevertheless, insensible choices of $\bo s$ yield useless moves.
If, for example, the image of $\bo s$ is a null set under $\pi$, the probability of transitioning anywhere with this move is zero.
Similarly, if the image of a non-null set poses a null set, this non-null set remains practically inaccessible by this move.

The following lemma represents a reconciliation of Lemma \ref{lem:post-hoc} and \ref{lem:metgibbs}.
\begin{lemma}[Siems]
	\label{lem:metgibbspost-hoc}	
	Assume that there exist a sigma finite measure space $(\ustates,\mcal{B},\nu)$ and measurable functions $\bo u:\states'\rightarrow \ustates$ and $f:\states\times \ustates\rightarrow\mathbb{R}_{\geq 0}$ such that 
	\begin{align*}
	&\textbf{\emph{(I)}}\ \ (\bo s, \bo u):\states'\rightarrow\states\times \ustates \text{ is a bijection}\\
	&\textbf{\emph{(II)}}\ \ \lambda'\left((\bo s, \bo u)^{-1}(ds, du)\right)=f(s, u)\lambda\otimes\nu(ds, du)
	\end{align*}
	For $s\in\states$, we may write
	\begin{align}
	\label{eq:metgibpost}
	\prob{\rand X\in \LargerCdot\mid \funcg(\rand X)=s}\propto
	\dens(s)\delta_{s}(\LargerCdot)
	+\int \delta_{(\bo s, \bo u)^{-1}}(\LargerCdot)\dens\left((\bo s, \bo u)^{-1}\right) f(s,u)\nu(du)
	\end{align}	
\end{lemma}
\begin{proof}
We have to show that $\int_B \prob{\rand X\in A\mid \funcg(\rand X)=s} \pi\circ \funcg^{-1}(ds)=\pi(A\cap \funcg^{-1}(B))$ for all $B\in\mcal{A}$ and $A\in\mcal{A}\otimes\mcal{A}'$.
To this end, consider that
\begin{align*}
&\pi\circ\funcg^{-1}=\int_{\LargerCdot}\dens(s)\lambda(ds)+\int\delta_{\bo s}(\LargerCdot)\dens(s')\lambda'(ds')\\
&=\int_{\LargerCdot}\dens(s)\lambda(ds)+\int\delta_{s}(\LargerCdot)\dens((\bo s, \bo u)^{-1})f(u,s)\nu\otimes\lambda(du, ds)\\
&=\int_{\LargerCdot}\dens(s)\lambda(ds)
+\int_\LargerCdot \underbrace{\int\dens\left((\bo s, \bo u)^{-1}\right)f(u,s)\nu(du)}_{=\bcc(s)}\lambda(ds)
=\int_{\LargerCdot}\dens(s)+\bcc(s)\lambda(ds)
\end{align*}
This allows us to write
\begin{align*}
&\int_B \prob{\rand X\in A\mid \funcg(\rand X)=s}\left(\dens(s)+\bcc(s)\right)\lambda(ds)\\
&\overset{*}{=}\int_B \dens(s)\delta_{s}(A)+ \int \delta_{(\bo s, \bo u)^{-1}}(A)\dens\left((\bo s, \bo u)^{-1}\right) f(s,u)\nu(du)\lambda(ds)\\
&=\pi(A\cap B)
+\int\delta_s(B)\delta_{(\bo s, \bo u)^{-1}}(A)\dens\left((\bo s, \bo u)^{-1}\right) \lambda'\left((\bo s, \bo u)^{-1}(ds, du)\right)\\
&=\pi(A\cap B)
+\int\delta_{\bo s}(B)\delta_{s'}(A)\dens(s') \lambda'(ds')
=\pi(A\cap B)+\pi(A\cap \bo s^{-1}(B))=\pi(A\cap \funcg^{-1}(B))
\end{align*}
whereby in $*$, we have set the proportionality constant in (\ref{eq:metgibpost}) to $1\slash\left(\dens(s)+\bcc(s)\right)$.
\end{proof}

Consequently, in accordance with Lemma \ref{lem:metgibbs}, the proposal for transitioning from $s'\in\states'$ to $\states$ complies with $\delta_{\funcg(s')}$ and thus applies $\bo s$ deterministically. 
In turn, as in post-hoc translations, the backward move may exploit the form of $\int \delta_{(\bo s, \bo u)^{-1}}(\LargerCdot)\nu(du)$.
To this end, we employ a kernel density $\densk$ from $(\states,\mcal{A})$ to $(\ustates,\mcal{B}, \nu)$.
A small contemplation yields that the acceptance probability for transitions from $s'\in\states'$ to $\states$ reads
\begin{align*}
\min\left\{1,\frac{\dens(\bo s)\densk(\bo s, \bo u)}
{\dens(s')f(\bo s, \bo u)}\cdot\frac{\moveprob(\bo s)}{\moveprob(s')}\right\}
\end{align*}
with move probability $\moveprob$.
This is the same acceptance probability as in the corresponding post-hoc approach.

Please note that the Metropolis within Gibbs algorithm doesn't capture all post-hoc translations.
For translations that are not SDT's, there is hardly a suitable $\funcg$ available that yields the same sampling scheme.
To see this, consider the relation defined by the viable transitions under the two translation functions that belong to a move and its backward move.
$\funcg$ is build from the partition of the two state spaces derived from the implied equivalence relation.
However, for non-SDT's this partition can easily contain the spaces as a whole and is then useless.



\section{A changepoint example}
\label{chap:cpexample}

\todo{convergence}

In this section, we scrutinize our theory by means of a small \cp example.
To this end, we consider three different implementations for so-called birth and death moves which either add or remove \cps.
Firstly, we employ very plain and unsophisticated proposals.
Secondly, we apply ad-hoc translations that incorporate maximizers of certain partial likelihoods.
Finally, we use tightly adapted post-hoc translations in the fashion of the reversible jump algorithm.

It turns out that the ad-hoc and post-hoc approaches perform equally well on this example in terms of their computation times and acceptance rates.
In contrast, the plain approach performs poorly, which justifies the need of sophisticated ad-hoc and post-hoc translations.

\begin{figure}[ht]
	\includegraphics[width=1\linewidth]{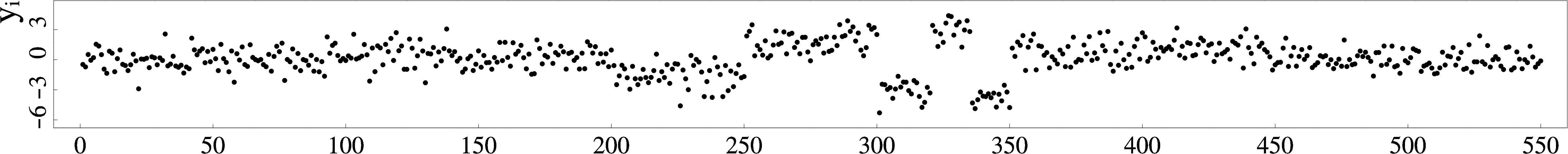}
	\caption{Gaussian change in expectation example dataset.}
	\label{pic:guassmhrj}
\end{figure}

Figure \ref{pic:guassmhrj} shows an artificial dataset.
The $n=550$ data points were drawn independently from a Gaussian distribution with variance 1 and expectations that where subject to 9 successive changes.
The overall 10 expectations and 9 \cp locations were chosen manually.

To build an exemplary Bayesian model here, we choose a prior for the changepoint locations and expectations: 
The time from one changepoint to the next is geometrically distributed with $q=3\slash550$ and the expectations in case of a jump and at the beginning are distributed according to $\mcal{N}(0, 25)$.
The data at time point $i=1,...,n$ also follows a Gaussian distribution with corresponding expectation at $i$ and variance equal to 1.

The sampling starts with no changepoints and an overall expectation of 0.
Subsequently, new changepoint locations may be found or discarded through \emph{birth} and \emph{death} moves.
Additionally, we employ \emph{shift} moves to shift single changepoints and \emph{adjust} moves to adjust single expectations.

Each move only manipulates a subset of the expectations or a single \cp and leaves the rest constant.
Whilst this could be implemented through post-hoc translations directly, ad-hoc translations do not support such moves on their own.
Either way, we apply the Metropolis within Gibbs framework since this is most convenient and requires no extra work.
Consequently, we build our proposals solely based on the quantities that are subject to a modification.

Please note, that the backward moves for adjust and shift are again adjust and shift.
Similarly, death and birth will be reversed by birth and death, respectively.
This arises naturally here, since each of the moves has its very own and unique complementary move, e.g. a birth can only be reversed by a death.

An adjust move relocates the old height $h'$ of a uniformly chosen segment according to $\mcal{N}(h', 0.5)$.
In turn, a shift move shifts the location of a uniformly chosen changepoint uniformly to a new position within the neighboring changepoints.

Let $\phi(x\semic \mu,\sigma^2)$ be the density of the univariate Gaussian distribution with expectation $\mu$ and variance $\sigma^2$ evaluated at $x$.
Adjusting the height $h'$ of a segment $S$ to $h$ yields an acceptance probability of
\begin{align*}
\min\left\{1,\quad \frac{\phi(h\semic 0,25)}{\phi(h'\semic 0,25)}\prod\limits_{j \in S}\frac{\phi(y_j\semic h, 1)}{\phi(y_j\semic h', 1)}\quad\right\}
\end{align*}
Note that the probabilities of choosing the shift move, the segment and the density values of the data points within the untouched segments, cancel out.
Furthermore, since the normal distribution is symmetric w.r.t. its expectation, the proposal densities cancel out as well.
What remains are the density values for the data points within the adjusted segment and the priors for the segment heights.

Consider $1$ and $n+1$ to be immutable auxiliary changepoints.
In the shift move, we shift an existing changepoint, say at location $i$, uniformly to a new location, say $j$, within the neighboring \cps.
Let $\ell$ and $k$ be the changepoints to the left and right of the changepoint at $i$ and let $h_1$ respectively $h_2$ be the corresponding segment heights.
The acceptance probability for the shift move reads
\begin{align*}
&\min\left\{1,\quad \frac{\prod\limits_{m=\ell}^{j-1} \phi(y_m\semic h_1, 1)\prod\limits_{m=j}^{k-1} \phi(y_m\semic h_2, 1)}{\prod\limits_{m=\ell}^{i-1} \phi(y_m\semic h_1, 1)\prod\limits_{m=i}^{k-1} \phi(y_m\semic h_2, 1)}\quad\right\}
\end{align*}
Here, the probabilities of choosing the move and the changepoint, the priors for the heights and the density values of the data within the untouched segments cancel out.

\subsection{A plain implementation of birth and death}
\label{chap:plainbd}
In the death move, we choose a changepoint uniformly, remove it and propose a new height for the remaining segment according to $\mcal{N}(0, 25)$.
In the birth move, we propose a new changepoint uniformly among the timepoints without a changepoint.
The two heights left and right of the new changepoint are then proposed independently according to $\mcal{N}(0, 25)$.

Let $p_d$ respectively $p_b$ be the probabilities to choose a death or a birth move, respectively.
We perform a death move.
Therefore, let $i$ be the timepoint of the changepoint that is to be removed and 
let $\ell$ and $k$ be the changepoints to the left and right of the changepoint at $i$ and let $h_1$ and $h_2$ be the corresponding segment heights.
Finally, let $h$ be the ensuing single segment height.
The acceptance probability reads
$\min\left\{1,\ \frac{n-c}{c+1}
\cdot\frac{p_b}{p_d}
\cdot L_{\ell, i, k}\right\}$,
whereby $c+1$ is the number of changepoints without $1$ and $n+1$ and 
\begin{align*}
L_{\ell,i, k}=\frac{\prod\limits_{m=\ell}^{k-1}\phi(y_m\semic h, 1)}{\prod\limits_{m=\ell}^{i-1} \phi(y_m\semic h_1, 1)\prod\limits_{m=i}^{k-1} \phi(y_m\semic h_2, 1)}
\end{align*}

In order to derive the acceptance probability for the corresponding birth move, we just take the reciprocal of the acceptance ratio.

\subsection{An ad-hoc implementation of birth and death}
\label{chap:ad-hocimpl}

Here, we utilize parts of the likelihood of the model.
The likelihood of the data within a single segment can be maximized w.r.t. its expectation just by choosing the mean of the involved data points.
Thus, by proposing new segment heights close to these means we might propose sensible states that yield good acceptance rates.
To this end, our sampling approach proposes new segment heights through a normal distribution with variance $0.01$ and expectation equal to the mean of the associated data.

Let $\tau^2=0.01$ and for $1\leq \ell\leq k\leq n$ define $\mu_{\ell:k}=\sum_{m=\ell}^{k-1}\frac{y_m}{(k-\ell)}$.
By using the notation of Section \ref{chap:plainbd}, the acceptance probability for the death move reads
\begin{align*}
\min\left\{1,\ 
\frac{n-c}{c+1}
\cdot\frac{p_b}{p_d}
\cdot L_{\ell,i, k}\cdot \frac{\phi(h\semic 0,25)}{\phi(h_1\semic 0,25)\phi(h_2\semic 0,25)}\cdot \frac{\phi(h_1\semic \mu_{l:i}, \tau^2)\phi(h_2\semic \mu_{i:k}, \tau^2)}
{\phi(h\semic \mu_{l:k}, \tau^2)}\quad\right\}
\end{align*}

\subsection{A post-hoc implementation of birth and death}
Now, we design birth and death moves by means of the reversible jump paradigm.
It is crucial hereby to find a carefully adapted diffeomorphism.
A good estimator for the expectation of the data within a segment is the mean of the observations belonging to it.
If $h$ is this mean and we split the segment, say from $\ell$ to $k$, into two parts through a \cp at $i$, such that the lengths of the resulting segments are $n_1$ and $n_2$ and their corresponding mean values are $h_1$ and $h_2$, we see that $h=\frac{n_1h_1+n_2h_2}{n_1+n_2}$.

Furthermore, in order to build a diffeomorphism thereof, we introduce an auxiliary variable $u$.
A sense of parsimony is achieved by stipulating that $u\in\mathbb{R}$.
We treat $u$ as part of the target space by setting $u=h_2$.
That way, $(h,u)$ describes a diffeomorphism with Jacobi determinant $\frac{n_1}{n_1+n_2}$.

The opportunity to embed the auxiliary variables directly into the target space often remains unnoticed in practical implementations. 
Instead, $u$ would be sampled independently which is prone to forfeit acceptance rate by proposing unsuited states.

This yields an SDT, whereby we rely on the reasoning of Section \ref{chap:ad-hocimpl} and draw $u$ from $\mcal{N}(m_{i,k}, \tau)$ in the birth move.
Therewith, the acceptance probability for the death move reads
$\min\left\{1,\ \frac{n_1}{n_1+n_2}
\cdot\frac{n-c}{c+1}\cdot\frac{p_b}{p_d}
\cdot L_{\ell,i, k}\cdot \frac{\phi(h\semic 0,25)}{\phi(h_1\semic 0,25)\phi(h_2\semic 0,25)}\cdot \phi(h_2, m_{i,k}, \tau) \right\}$

\subsection{Results and conclusions}
Now, we want to compare our proposals empirically in terms of their runtime and acceptance rates.
To this end, we set all four move probabilities to $0.25$, except in the boundary cases.
If there is only one segment, birth and adjust have the same probability of 0.75 and if each timepoint accommodates a changepoint then death and adjust have the probabilities 0.5 respectively 0.25.

\begin{table}[ht]
	\begin{center}	
		\begin{tabular}{ l || c | c | c | c}
			Acceptance rates & Death move & Birth move & Shift move & Adjust move   \\
			\hline\hline
			Plain proposals & 0.0021 & 0.0022 & 0.0681 & 0.2896\\
			Ad-hoc translations & 0.0588 & 0.0594 & 0.0678 & 0.2904\\			
			Post-hoc translations & 0.0639 & 0.0645 & 0.0681 & 0.2899\\
		\end{tabular}
	\end{center}
	\caption{Acceptance rates for particular moves.}
	\label{fig:fractable}
\end{table}

We computed $10^{7}$ successive links for each of the three Metropolis-Hastings approaches.
On my computer (AMD 64 with a 3200 GHZ CPU), the sampling algorithm for the plain birth and death proposals and the post-hoc translations needed 12 seconds and the ad-hoc approach took 13 seconds.

Table \ref{fig:fractable} shows the acceptance rates for the individual moves.
It becomes apparent, that the plain birth and death moves struggle with finding alternate changepoint configurations.
Thus, the Markov chain performs poorly in exploring the state space, which impairs the mixing time.

Our ad-hoc and post-hoc translations compare well with each other. 
Though, both exhibit very different strategies in proposing new segment heights.
Furthermore, several experiments revealed that an acceptance rate for birth and death moves of more than 6.5\% is infeasible as long as the new \cp locations are chosen in a plain uniform manner.
Thus, the achieved rates are sound, however, higher rates are usually considered as better \citep{bedard2008optimal}.

It became clear that sophisticated approaches like ad-hoc and post-hoc translations are key for sampling within mixed spaces.
Whilst our post-hoc proposals can be regarded as ambitiously developed, the ad-hoc proposals constitute a simple independence sampler that was derived from straightforward ideas.
This indicates the great potential of ad-hoc translations.
In turn, post-hoc translations help by reducing the number of random variables to be generated and may hence result in more tightly adapted proposals that exhibit a reduced computational runtime.

\section{The Metropolis-Hastings algorithm in abstract terms}
\label{chap:genmh}
So far, we have considered different strategies to implement the Metropolis-Hastings algorithm based on densities and measure spaces.
\cite{tierney1998}, in turn, provides theoretic foundations for the Metropolis-Hastings algorithm to work with $\pi$ and $\dist$ directly.
At the same time, it introduces the notion of a maximal Metropolis-Hastings algorithm.
Even though this concept remained widely unrecognized, it gives important advice on the design of proposals and allows for a classification of Metropolis-Hastings approaches.

In this section, we will summarize the main outcomes of \cite{tierney1998} and put them in line with the results of \cite{roodaki2011} for mixture proposals.
Subsequently, we examine maximality with regards to our own methods.

Given a measurable space $(\states,\mcal{A})$, for a set $A\in\mcal{A}\otimes \mcal{A}$ we define the \emph{transpose} of $A$ as $A^t=\{(s,s')\mid (s',s)\in A\}$.
We leave it to the reader to show that $A^t\in\mcal{A}\otimes \mcal{A}$ and that $\pi\otimes\dist(\LargerCdot)=\pi\otimes\dist(\LargerCdot^t)$ if and only if $\dist$ preserves the detailed balance of w.r.t. $\pi$.

In the following, we are given an arbitrary probability space $(\states,\mcal{A},\pi)$ and a Markov kernel $\dist$ from $(\states,\mcal{A})$ to $(\states,\mcal{A})$.
The idea is to use $\acc_{s' s}=\min\{1,\ratio_{s's}\}$ with $\ratio_{s's}=\frac{\pi\otimes\dist(ds,ds')}{\pi\otimes\dist(ds',ds)}$ as the acceptance probability in the Metropolis-Hastings algorithm.
$\ratio_{s's}$ is the density of $\pi\otimes\dist(\LargerCdot^t)$ w.r.t. $\pi\otimes\dist$ evaluated at $(s',s)\in\states\times\states$.
It is important to note that $\ratio_{s's}$ may not exist everywhere.
Furthermore, in order to prove detailed balance in the same fashion as before, we require that $\ratio_{s's}=1\slash\ratio_{ss'}$.

Let $\xi$ be a symmetric measure over $\mcal{A}\otimes \mcal{A}$ such that $\pi\otimes\dist$ and $\pi\otimes\dist(\LargerCdot^t)$ are absolutely continuous w.r.t. $\xi$, for example $\xi=\pi\otimes\dist+\pi\otimes\dist(\LargerCdot^t)$.
The theorem of Radon-Nikodym \citep{nikodym1930generalisation} ensures that the densities $\frac{\pi\otimes\dist(ds',ds)}{\xi(ds',ds)}$ and $\frac{\pi\otimes\dist(ds,ds')}{\xi(ds',ds)}$ exist.
\begin{definition}
	Let 
	\begin{align*}
	R=\Bigg\{(s',s)\bbmid\frac{\pi\otimes\dist(ds',ds)}{\xi(ds',ds)}>0, \frac{\pi\otimes\dist(ds,ds')}{\xi(ds',ds)}>0\Bigg\} 
	\end{align*}
\end{definition}

Thus, on $R\cap\mcal{A}\otimes\mcal{A}$ the density $\ratio_{s's}$ exist and equals the ratio of $\frac{\pi\otimes\dist(ds,ds')}{\xi(ds',ds)}$ and $\frac{\pi\otimes\dist(ds',ds)}{\xi(ds',ds)}$, and $\ratio_{s's}=1\slash\ratio_{ss'}$ holds.
This yields a Metropolis-Hastings algorithm in the fashion of Lemma \ref{lem:mh}.

\begin{lemma}[\cite{tierney1998}]
	\label{lem:genmh}	
	The transition kernel $\mu$ in (\ref{eq:mhtranskern}) with acceptance probability 
	\begin{align}
	\label{eq:accratio}
	\acc_{s' s}=\begin{cases}
	\min\{1,\ratio_{s's}\}& (s',s)\in R\\
	0& \text{otherwise}
	\end{cases}
	\end{align}
	preserves the detailed balance w.r.t. $\pi$.
	Furthermore, $\acc_{s' s}$ is $\pi\otimes\dist(\LargerCdot\cap R)$ almost surely maximal over all acceptance probabilities that preserve the detailed balance w.r.t. $\pi$.
\end{lemma}
\begin{proof}
	The detailed balance can be proven in the same fashion as Lemma \ref{lem:mh}.	
	From \cite{tierney1998}, we know that any suitable acceptance probability, $\hat \acc_{ss'}$, satisfies $\ratio_{s s'}\hat\acc_{s's}=\hat\acc_{ss'}$ almost surely w.r.t. $\pi\otimes\dist(\LargerCdot\cap R)$ and thus $\hat\acc_{ss'}=\ratio_{s' s}\hat\acc_{s's}\leq\min\{1, \ratio_{s' s}\}$.
\end{proof}

\begin{myremark}
In accordance with \cite{tierney1998}, the set $R$ is symmetric and uniquely defined up to null sets w.r.t. $\pi\otimes\dist$ and $\pi\otimes\dist(\LargerCdot^t)$.
It represents the set of possible transitions in forward and backward direction.
Furthermore, the density $\ratio_{s's}$ is $\pi\otimes\dist(\LargerCdot\cap R)$ almost surely uniquely defined.
Thus, we can speak of the \emph{maximal Metropolis-Hastings algorithm} for the proposal $\dist$ when we use $\acc_{s' s}$ as in (\ref{eq:accratio}).
\end{myremark}

The above formulation of the Metropolis-Hastings algorithm allows for the employment of arbitrary proposal distributions.
Remarkably, no underlying measure space or kernel density is required.
However, $R$ narrows down the set of possible translations.
In the extreme case, when it is empty, no transitions we will be accepted whatsoever.

In the following, in order to keep the notation uncluttered, we do not explicitly mention the set $R$ in the acceptance probability anymore.
Instead, we implicitly set the densities to zero wherever they are not defined.

\subsection{Mixture proposals in general and maximality in particular}
\label{chap:mixtures}

So far, we have build the acceptance probabilities for mixture proposals from unique pairs of forward and backward moves, i.e. $\ell$ and $\ellr$.
However, the computation of the acceptance probability for the maximal Metropolis-Hastings algorithm deploys the mixture proposal as a whole.
Consequently, if the two approaches are different, the pairwise one is prone to yield smaller acceptance rates (see also \cite{tierney1998}). 
Unfortunately, lower acceptance rates increase the dependencies of successive links, which, in turn, impairs the mixing time.

Nevertheless, the pairwise approach is an important measure of convenience.
It exonerates us from evaluating the full range of component proposals at each step.
This gives rise to the following lemma which provides sufficient conditions for the pairwise approach to satisfy the maximality.

\begin{lemma}[\cite{roodaki2011}]
	\label{lem:mixmh}
	If for all $\ell\in\moves$, there exist disjoint $Z_{\ell}\in\mcal{A}\otimes\mcal{A}$ with
	\begin{align*}
	&\textbf{(I) }\quad\int_{\states\times \states \backslash Z_\ell}\moveprob_\ell(s')\pi\otimes\dist_\ell(ds',ds)=0\\
	&\textbf{(II) }\quad Z_{\ellr}=Z_\ell^t
	\end{align*}	
	we get 
	\begin{align}
	\label{eq:mixpairdens}
	&\frac{\pi\otimes\dist(ds, ds')}{\pi\otimes\dist(ds', ds)}=\frac{\moveprob_{\ellr}(s)}{\moveprob_\ell(s')}\frac{\pi\otimes\dist_{\ellr}(ds, ds')}{\pi\otimes\dist_\ell(ds', ds)}	
	\end{align}	
	for $(s',s)\in Z_\ell$.
\end{lemma} 
\begin{proof}
	Define $\ratio_{s's}^\ell=\frac{\moveprob_{\ellr}(s)}{\moveprob_\ell(s')}\frac{\pi\otimes\dist_{\ellr}(ds, ds')}{\pi\otimes\dist_\ell(ds', ds)}$. 
	For $(s',s)\in Z_\ell$ we get
	\begin{align*}
	&\ratio_{s's}^\ell\pi\otimes \dist(ds',ds)
	=\ratio_{s's}^\ell\moveprob_\ell(s')\pi\otimes \dist_\ell(ds',ds)
	=\moveprob_{\ellr}(s)\dist_{\ellr}(ds,ds')=\pi\otimes \dist(ds,ds')
	\end{align*}
	This holds since \textbf{(I)} and \textbf{(II)} state that we can set $\moveprob_k(s')\pi\otimes \dist_k(ds',ds)=\moveprob_\kr(s)\pi\otimes \dist_\kr(ds,ds')=0$ for all $k\not =\ell$.
\end{proof}

Having said that, \cite{roodaki2011} wanted to make a different statement with their theorem.
They tried to elaborate upon when it is possible to perform the Metropolis-Hastings algorithm solely based on pairs of moves.
Therewith, they appear to have missed that this is always viable as long as the pairs are unique \citep{tierney1998}.

Under the assumption that the conditions \textbf{(I)} and \textbf{(II)} of Lemma \ref{lem:mixmh} are met, we may ask for the maximality of the approaches considered in this paper.
In most cases, it is sufficient to show that for each move $\ell$ the acceptance ratio of our acceptance probability corresponds to the l.h.s. of (\ref{eq:mixpairdens}).
This is straightforward for the primal Metropolis-Hastings algorithm and ad-hoc translations, and also for (\ref{eq:accdiscpost}) and (\ref{eq:metgibbsacc}).

However, post-hoc translations that accept within the auxiliary space are deviant.
In the context of Lemma \ref{lem:post-hoc}, we define 
\begin{align*}
\ratio_{s' u'}^\ell=\frac{\dens(\bo s_{\ell})\densk_{\ellr}((\bo s, \bo u)_{\ell})}{\dens(s')\densk_{\ell}(s',u')}\cdot f_{\ell}(s',u')
\end{align*}
Furthermore, let $\dist_{\ell}(s',\LargerCdot)=\int \delta_{\bo s_{\ell}}(\LargerCdot)\densk_{\ell}(s', u') \nu_\ell(du')$ be the proposal of move $\ell$.

\begin{lemma}[Siems]
\label{lem:rvjproposdens}
For move $\ell$ with transitions from $s'\in\states_j$ to $s\in\states_i$, we get
\begin{align*}
\frac{\pi\otimes\dist_{\ellr}(ds,ds')}{\pi\otimes\dist_{\ell}(ds',ds)}=\expec{\ratio_{s' \rand U}^\ell\bmid \bo s_{\ell}(s',\rand U)=s}
\end{align*}
whereby $\rand U$ is a random variable distributed according to $\int_\LargerCdot\densk_{\ell}(s', u')\nu_\ell(du')$.
\end{lemma}
\begin{proof}
This holds since for $A\in\mathcal{A}_j$ and $B\in\mathcal{A}_i$
\begin{align*}
&\int_A \int_B \expec{\ratio_{s' \rand U}^\ell\bmid \bo s_{\ell}(s',\rand U)=s}\dist_{\ell}(s',ds)\pi(ds')\\
&=\int_A \int \expec{\ratio_{s' \rand U}^\ell\bmid \bo s_{\ell}(s',\rand U)=\bo{s}_{\ell}(s', u')}\delta_{\bo{s}_{\ell}}(B)\densk_{\ell}(s',u')\nu_\ell(du')\pi(ds')\\
&\overset{*}{=}\int_A \int \delta_{\bo{s}_{\ell}}(B)\ratio_{s' u'}^\ell\densk_{\ell}(s',u')\dens(s')\nu_\ell(du')\lambda(ds')\\
&=\int_A\delta_{\bo s_{\ell}}(B) \int_B \dens(\bo s_{\ell})\densk_{\ellr}((\bo s, \bo u)_{\ell})f_{\ell}(s',u')\nu_\ell(du')\lambda(ds')\\
&\overset{**}{=}\int_B \int \delta_{\bo s_{\ellr}}(A)\densk_{\ellr}(s,u)\dens(s)\nu_\ellr(du)\lambda(ds)=\pi\otimes\dist_{\ellr}(B\times A)
\end{align*}
whereby we have used the basic properties of the conditional expectation in $*$ and condition \textbf{(I)} and \textbf{(II)} of Lemma \ref{lem:post-hoc} in **.
\end{proof}

Thus, the acceptance probability for move $\ell$ of the maximal Metropolis-Hastings algorithm for the post-hoc translations of Lemma \ref{lem:post-hoc} reads
\begin{align*}
\acc_{s's}^\ell=\min\left\{1,\ \expec{\ratio_{s' \rand U}^\ell\cdot \frac{\moveprob_{\ellr}(s)}{\moveprob_{\ell}(s')}\bmid \bo s_{\ell}(s',\rand U)=s}\right\}
\end{align*}
This states that given a previous link $s'$, we accept or reject the newly proposed state $s$ based on an expectation drawn from $\ratio_{s' u'}$ over those values of $u'$ that yield $\bo s_{ji}(s',u')=s$.
If $\bo s_{\ell}(s',\LargerCdot)$ is injective for all $\ell$, $\ratio_{s' u'}^\ell=\expec{\ratio_{s' \rand U}^\ell\bmid \bo s_{\ell}(s',\rand U)=\bo s_{ji}(s',u')}$ holds and the post-hoc approach is maximal.
This is guaranteed for SDT's in the fashion of Lemma \ref{lem:post-hoc} and \ref{lem:metgibbspost-hoc}.

On the downside, consider a transition $(s',s)$ where the maximal Metropolis-Hastings algorithm exhibits an acceptance probability of one, i.e. $\int \ratio_{s' u'}\prob{\rand U\in du'\mid \bo s_{ji}(s',\rand U)=s}\geq 1$.
The corresponding post-hoc translation approach accepts $(s',s)$ with probability one if and only if $1=\prob{\ratio_{s' \rand U}\geq 1\mid \bo s_{ji}(s',\rand U)=s}$.
In non-trivial cases, these are two very distinct conditions that may lead to significantly different results.

This suggests that post-hoc translations in the fashion of Lemma \ref{lem:post-hoc} should preferably employ translations that meet the mentioned injectivity.  
This highlights the expediency of SDT's. 
However, in some particular cases, it might be reasonable to give up on maximality in order to address tractability issues with compound proposals (see also Section \ref{chap:mixspaces}).

\section{Discussion and literature review}
\label{chap:discussion}

In this paper, we examine several practical and theoretical approaches to derive acceptance probabilities for the Metropolis-Hastings algorithm.
The common challenge hereby is the derivation of compatible densities for the transitions in forward and backward direction.
While this remains invisible in simple cases, sampling across collections of different measure spaces often requires an explicit treatment.

Ad-hoc translations pose the easiest example since they just require the individual proposals to comply with the measure of their target space.
In turn, post-hoc translations demands a more elaborate integral transformation to be viable.
Finally, Metropolis within Gibbs calls for reasonable representations of conditional versions of the target distribution.

The consideration of the maximal Metropolis-Hastings algorithm of \cite{tierney1998} revealed that almost all our methods are of the same kind.
Only post-hoc translations that accept and reject within an auxiliary space are different if a certain injectivity condition is not met.


All considered algorithms make intense use of mixture proposals.
They feature a sophisticated design in a modular fashion and are therefore of high practical value.
Furthermore, it is beneficial to compute the acceptance probabilities solely based on unique pairs of the moves.
However, an overlap in their support may forfeit maximality here.

So far, we have only considered mixture proposals w.r.t. a countable set of moves.
However, the set of moves can alternatively be part of any measure space, whereby the move probabilities are replaced by a kernel density.
Consequently, summation over mixture components is replaced by integration.
As before, we need to define unique pairs of forward and backward moves.
An analogy to this is a post-hoc translation having an auxiliary variable that represents the move with an associated auxiliary function that determines the corresponding backward move.

Ultimately, our purposive contemplation of ad-hoc as well as post-hoc translations, and Metropolis within Gibbs opens up innovative ways of choosing sampling schemes that satisfy various demands regarding simplicity and sophistication.
Furthermore, we help resolving the ubiquitous confusions around the reversible jump framework.

In the remainder of this section, we discuss a sequence of far-reaching errors in reasoning, which all agree that ordinary MCMC methods struggle with trans dimensional sampling.

\cite{green1995} concludes wrongly that we have to use dimension matching  in order to pursue trans dimensional moves.
On page 715 in \cite{green1995}, the acceptance probability solely utilizes single components of the mixture proposal which implies that the density used to perform a move must also be used for the corresponding backward move.
For trans dimensional moves, this is obviously not possible and made it inevitably necessary to introduce the dimension matching condition.
Thus, Green's undeniably great idea appears to be the result of a simple error in reasoning.

\cite{carlin1995} claim, by referring to \cite{tierney1994}, that trans dimensional moves create absorbing states and therefore violate the convergence of Markov chains.
Unfortunately, \cite{tierney1994} doesn't seem to provide this statement.
In order to circumvent this alleged problem, they propose to work on the product space $\prod_{i\in\spaces}\states_i$ instead.
Frankly speaking, this adds a huge burden just to avoid the inhomogeneous nature of the spaces.
According to Google schoolar, \cite{carlin1995} was cited over 1000 times, though a few of them question the necessity of this approach, e.g. \cite{godsill2001, green2003, green2009, green1998model}.

In a more measure theoretic context, \cite{chen2012} comes to the conclusion that usual MCMC cannot transition across spaces of different dimensions.
They argue with the lack of dominating measures and therewith ignore the fact that each countable collection of sigma finite measures indeed exhibits a common dominating measure.

As we have seen, the acceptance probability of the reversible jump algorithm is not necessarily maximal in the sense of Lemma \ref{lem:genmh}.
Interestingly, \cite{green1995} still refers to \cite{peskun1973optimum} in order underline the maximality of his choice of acceptance probability.

These misperceptions have introduced a significant burden to the field of trans dimensional sampling.
The concerning literature is difficult to overview due to inconsistent claims and full of poor mathematical language as a result of the absurd overemphasis on dimensionality.

Furthermore, the name "reversible jump" stems from Green's wrong conviction that he constructed a method which overcomes the non-reversibility of trans dimensional moves in the Metropolis-Hastings algorithm.
Thus, this name lacks a proper meaning and collides awkwardly with "reversible MCMC" making this topic even harder to grasp for newcomers.

Unfortunately, due to the outshining popularity of the reversible jump algorithm, the unprepared user of mixed spaces is forced to understand and use diffeomorphisms and their Jacobi determinants.
Furthermore, the diffeomorphisms need to fit very closely to the employed model and application in order to be rewarding.
However, it is by no means clear if this is per se appropriate in all application scenarios.


\bibliographystyle{chicago}      
\bibliography{../bibtex}   


\end{document}